\documentclass[12pt]
{amsart}
\usepackage{amsfonts,amsmath,amssymb,amsthm,graphicx,afterpage,url}
\usepackage{hyperref}

\input{xy}
\xyoption{all}

\def\Int{\mathop{\fam0 Int}}

\def\R{{\mathbb R}}

\long\def\comment#1\endcomment{}

\newcommand{\jonly}[1]{}
\newcommand{\aronly}[1]{#1}

    \theoremstyle{theorem}
         \newtheorem{theorem}{Theorem}
         \newtheorem{lemma}[theorem]{Lemma}

    \theoremstyle{definition}
         \newtheorem{remark}[theorem]{Remark}

\begin{document}

\title{A short exposition of S. Parsa's theorems on intrinsic linking and non-realizability}

\author{A. Skopenkov}

\thanks{Supported in part by the Russian Foundation for Basic Research Grant No. 19-01-00169.
\newline
I am grateful to A. Kupavskii, S. Parsa, M. Skopenkov and M. Tancer for useful discussions.
\newline
This is full version of a paper published in Discr. Comp. Geom.
}


\date{}

\maketitle

\abstract
We present a short exposition of the following results by S. Parsa.

{\it Let $L$ be a graph such that the join $L*\{1,2,3\}$ (i.e. the union of three cones over $L$ along their common bases) piecewise linearly (PL) embeds into $\mathbb R^4$.
Then  $L$  admits a PL embedding into $\mathbb R^3$ such that any two disjoint cycles have zero linking number.}

{\it There is $C$ such that every 2-dimensional simplicial complex having $n$ vertices and embeddable into $\mathbb R^4$ contains less than $Cn^{8/3}$ simplices of dimension 2.}

We also present
the analogue of the second result for intrinsic linking.
\endabstract

\bigskip
This paper provides short proofs of Theorems \ref{t:parsa}, \ref{c:parsa} and \ref{t:flosko} below.
Let $[k]:=\{1,\ldots,k\}$.

\begin{theorem}[see Remark \ref{r:sparsa}]\label{t:parsa} Let $L$ be a graph such that the join $L*[3]$
(i.e. the union of three cones over $L$ along their common bases) piecewise linearly (PL) embeds into $\R^4$.
Then  $L$  admits a PL embedding into $\R^3$ such that any two disjoint cycles have zero linking number.
\end{theorem}

\begin{proof}
Consider  $L*[3]$ as a subcomplex of some triangulation of $\R^4$.
Then there is a small general position 4-dimensional PL ball $\Delta^4$ containing the point $\emptyset*1\in\R^4$.
Hence the intersection $\partial\Delta^4\cap(L*[3])$ is PL homeomorphic to $L$.
Let us prove that this very embedding of $L$ into the 3-dimensional sphere $\partial\Delta^4$ satisfies the required property.

Take any two disjoint oriented closed polygonal lines $\beta,\gamma\subset\partial\Delta^4\cap(L*[3])\cong L$.
Then $(\beta* \{1,2\})-\Int\Delta^4$ and  $(\gamma*\{1,3\}) -\Int\Delta^4$
are two disjoint 2-dimensional PL disks in $\R^4-\Delta^4$ whose boundaries are $\beta$ and $\gamma$.
Hence $\beta$ and $\gamma$ have zero linking number in the 3-dimensional sphere $\partial\Delta^4$
(by the following well-known lemma applied to the 4-dimensional ball  $S^4-\Int\Delta^4$).
\end{proof}


\begin{lemma}\label{l:3and4}
If two disjoint oriented closed polygonal lines in the 3-dimensional sphere $\partial D^4$ bound two disjoint 2-dimensional PL disks in the 4-dimensional ball $D^4$,
then the polygonal lines have zero integer linking number in $\partial D^4$.
\end{lemma}

\begin{proof}\footnote{This proof is well-known.
For elementary versions of this well-known lemma see e.g. \cite[Lemmas 3.5 and 3.11]{Sk14}.}
Denote by $B,\Gamma\subset D^4$ the two disjoint oriented disks bounded by the polygonal lines
$\beta,\gamma\subset\partial D^4$.
Denote by $B'\subset\R^4-D^4$ an oriented disk (e.g. a cone) bounded by $\beta$.
Denote by $\Gamma'\subset\partial D^4$ a general position oriented singular disk (e.g. singular cone) bounded by $\gamma$.
We have $\beta\cap\Gamma'=(B\cup B')\cap(\Gamma\cup\Gamma')$.
Denote by the same letters integer chains carried by $\beta,\gamma,B,\Gamma$.
Denote by $\cdot_M$ the algebraic intersection of integer chains in $M$.
Then by general position the linking number of $\beta$ and $\gamma$ is
$\beta\cdot_{\partial D^4}\Gamma'=(B-B')\cdot_{\R^4}(\Gamma-\Gamma')=0$.
\end{proof}

\begin{remark}\label{r:sparsa}
(a) Theorem \ref{t:parsa} trivially generalizes to a $d$-dimensional finite simplicial complex $L$ and embeddings $L*[3]\to\R^{2d+2}$, $L\to\R^{2d+1}$.
For $d\ne2$ and a $d$-complex $L$ embeddability of $L*[3]$ into $\R^{2d+2}$ even implies embeddability of $L$ into $\R^{2d}$.
For the case $d=1$ considered in Theorem \ref{t:parsa}
this improvement follows from a theorem of Gr\"unbaum \cite{Gr69} (whose proof is more complicated).
For the case $d\ge3$ this improvement is proved in
\cite[$(iv)\Rightarrow(i)$ of Corollary 4.4]{MS06}, \cite{Pa20a}, \cite{PS20}.

(b) Theorem \ref{t:parsa} is formally a corollary of \cite[Theorem 1]{Pa15} but is essentially a restatement of \cite[Theorem 1]{Pa15} accessible to non-specialists.
In spite of being much shorter, the above proof of Theorem \ref{t:parsa} is not an alternative proof  comparatively to \cite[\S3]{Pa15} but is just a different exposition avoiding sophisticated language.
The above proof of Theorem \ref{t:parsa} is analogous to \cite[Example 2]{Sk03} where a relation between intrinsic linking in 3-space and non-realizability in 4-space was found and used.
Although the proof is simple, it easily generalizes to non-trivial results like a simple solution of the Menger 1929 conjecture and its generalizations \cite[Example 2, Lemmas 2 and 1']{Sk03}, see survey \cite{Sk14}.
\end{remark}




The following Theorem \ref{c:parsa}.a is a higher-dimensional generalization of upper estimation on the number of edges in a planar graph.

An embedding of a simplicial complex into $\R^{2d+1}$ is called {\bf linkless} if the images of any two $d$-dimensional spheres have zero linking number.

\begin{theorem}\label{c:parsa} (S. Parsa) (a) For every $d$ there is $C$ such that for every $n$ every $d$-dimensional simplicial complex having $n$ vertices and embeddable into $\R^{2d}$ contains less than
$Cn^{d+1-3^{1-d}}$ simplices of dimension $d$.

(b) For every $d$ there is $C$ such that for every $n$ every $d$-dimensional simplicial complex having $n$ vertices and linklessly embeddable into $\R^{2d+1}$ contains less than
$Cn^{d+1-4^{1-d}}$ simplices of dimension $d$.
\end{theorem}


The result (a) improves analogous result with $Cn^{d+1-3^{-d}}$ \cite{De93} and is covered by the Gr\"unbaum-Kalai-Sarkaria conjecture (whose proof is announced in \cite{Ad18}; I did not check that proof).
See  \cite[Theorems 3 and 4]{Pa15} and \cite{Pa20}.

The Flores 1934 Theorem states that the $d+1$ join power $[3]^{*(d+1)}=[3]*\ldots*[3]$ ($d+1$ copies of $[3]$) is not (PL or topologically) embeddable into $\mathbb R^{2d}$ \cite[\S5]{Ma03}.
(We have $[3]^{*2}=K_{3,3}$, so the case $d=1$ is even more classical.)
The $d+1$ join power $[4]^{*(d+1)}$ is not linklessly embeddable into
$\mathbb R^{2d+1}$ \cite[Lemma 1]{Sk03}.
(We have $[4]^{*2}=K_{4,4}$, so the case $d=1$ is due to Sachs.)
Theorem \ref{c:parsa} is implied by these results and the following theorem.

\begin{theorem}\label{t:flosko} For every $d,r$ there is $C$ such that for every $n$ every $d$-dimensional simplicial complex having $n$ vertices and not containing a subcomplex homeomorphic to $[r]^{*(d+1)}$ contains less than $Cn^{d+1-r^{1-d}}$ simplices of dimension $d$.
\end{theorem}

Proof of Theorem \ref{t:flosko} is based on the following lemma similar to the estimation of the number of edges in a graph not containing $K_{s+1,a}$ (Kovari-Sos-Tura\'n Theorem).

\begin{lemma}\label{l:parsa} For every integers $r,m,a,s$ and subsets $S_1,\ldots,S_m\subset[a]$ every whose $r$-tuple intersection contains at most $s$ elements we have
$$|S_1|+\ldots+|S_m|\le r(ma^{1-1/r}s^{1/r}+a).$$
\end{lemma}



The case $r=3$ of Theorem \ref{t:flosko} and of Lemma \ref{l:parsa} is essentially proved in \cite[\S3]{Pa15},
see \cite{Pa20}.
The case of arbitrary $r$ is analogous.

\begin{proof}[Proof of Lemma \ref{l:parsa}]
Denote by $d_q$ the number of subsets among $S_1,\ldots,S_m$ containing element $q\in[a]$.
We may assume that there is $\nu\le a$ such that $d_q\ge r$ when $q\le\nu$ and $d_q<r$ when $q>\nu$.
Then the required  inequality follows by
$$
\sum_{j=1}^m|S_j| = \sum_{q=1}^a d_q < ra+\sum_{q=1}^\nu d_q\quad\text{and}
$$
$$
\left(\sum_{q=1}^\nu d_q\right)^r\ \overset{(1)}\le\ \nu^{r-1}\sum_{q=1}^\nu d_q^r
\ \overset{(2)}\le\  r^r\nu^{r-1} \sum_{q=1}^\nu{d_q\choose r}\ \le\ r^ra^{r-1} \sum_{q=1}^a{d_q\choose r}
\ \overset{(3)}=
$$
$$
\overset{(3)}=\ r^ra^{r-1}\sum\limits_{\{j_1,\ldots,j_r\}}|S_{j_1}\cap\ldots\cap S_{j_r}|
\ \overset{(4)}\le\ r^ra^{r-1}s{m\choose r}
\ <\ r^rm^ra^{r-1}s.
$$
Here

$\bullet$ the inequality (1) is the inequality between the arithmetic mean and the degree $r$ mean;

$\bullet$ since $d_q\ge r$ when $q\le\nu$, the inequality (2) follows by
$$r^r{d_q\choose r} = \frac{r^rd_q^r}{r!}\left(1-\frac1{d_q}\right)\left(1-\frac2{d_q}\right)\ldots\left(1-\frac{r-1}{d_q}\right) \ge \frac{r^rd_q^r}{r!}\frac{r-1}r\frac{r-2}r\ldots\frac1r=d_q^r;$$

$\bullet$ the (in)equalities (3) and (4) are obtained by double counting the number of pairs $(\{j_1,\ldots,j_r\},q)$ of an $r$-element subset of $[m]$ and $q\in S_{j_1}\cap\ldots\cap S_{j_r}$.
\end{proof}



\begin{proof}[Proof of Theorem \ref{t:flosko}]
Induction on $d$.
The base $d=1$ follows because if a graph on $n$ vertices does not contain a subgraph homeomorphic to $K_{r,r}$, then the graph does not contain a subgraph homeomorphic to $K_{2r}$ and hence has $O(n)$ vertices \cite{BT98} (this  was apparently proved in the paper \cite{Ma68} which is not easily available to me).


Let us prove the inductive step.
If a $d$-dimensional simplicial complex $K$ having $n$ vertices does not contain a subcomplex homeomorphic to $[r]^{*(d+1)}$, then any $r$-tuple intersection of the links of vertices from $K$ does not contain a subcomplex homeomorphic to $[r]^{*d}$.
Apply Lemma \ref{l:parsa} to the set of $a\le{n\choose d}<n^d$ simplices of $K$ having dimension $d-1$,
and to $m=n$ subsets defined by links of the vertices.
By the inductive hypothesis $s\le Cn^{d-r^{2-d}}$.
Hence the number of $d$-simplices of $K$ does not exceed
$rnn^{(r-1)d/r}\left(Cn^{d-r^{2-d}}\right)^{1/r}=C'n^{d+1-r^{1-d}}$.
\end{proof}



The following version of Lemma \ref{l:parsa} is also possibly known.
It was (re)invented by I. Mitrofanov and the author in discussions of the $r$-fold Khintchine recurrence theorem,
see \cite[Problem 5]{OC}.

\begin{lemma}
\label{l:misk} For every integers $r,m,a$ and subsets $S_1,\ldots,S_m\subset[a]$ we have
$$a^{r-1}\sum_{j_1,\ldots,j_r=1}^m |S_{j_1}\cap\ldots\cap S_{j_r}| \ge \left(\sum_{j=1}^m|S_j|\right)^r.$$
\end{lemma}

\begin{proof}
Consider the decomposition of $[a]$ by the sets $S_j$ and their complements.
The sets of this decomposition correspond to subsets of $[m]$.
Denote by $\mu_A$ the number of elements in the set of this decomposition corresponding to a subset $A\subset[m]$.

To every pair $(A,j)$ of a subset $A\subset[m]$ and a number $j\in [m]$ assign 0 if $j\not\in A$ and assign $\mu_A$ if $j\in A$.
Let us double count the sum $\Sigma$ of the obtained $2^m\cdot m$ numbers.
We obtain
$$\sum\limits_{j=1}^m|S_j|=\Sigma=\sum\limits_{A\subset[m]}|A|\mu_A.$$
To every pair $(A,(j_1,\ldots,j_r))$ of a subset $A\subset[m]$ and a vector $(j_1,\ldots,j_r)\in [m]^r$ assign 0 if $\{j_1,\ldots,j_r\}\not\subset A$ and assign $\mu_A$ if $\{j_1,\ldots,j_r\}\subset A$.
Let us double count the sum $\Sigma_r$ of the obtained $2^m\cdot m^r$ numbers.
We obtain
$$\sum\limits_{j_1,\ldots,j_r=1}^m|S_{j_1}\cap\ldots\cap S_{j_r}|=\Sigma_r=\sum\limits_{A\subset[m]}|A|^r\mu_A.$$
Hence by the inequality between the weighted arithmetic mean and the weighted degree $r$ mean, and using
$\sum_{A\subset[m]}\mu_A=a$, we obtain
$$a^{r-1}\Sigma_r \ge \left(\sum\limits_{A\subset[m]}|A|\mu_A\right)^r = \left(\sum\limits_{j=1}^n|S_j|\right)^r.$$
\end{proof}

{\it Books, surveys and expository papers in this list are marked by the stars.}

\end{document}